\documentclass[12pt]{amsart}
\usepackage{amsmath,amsthm,amsfonts,latexsym,amssymb,amscd,color}
\usepackage{chemarr}
\usepackage{mathtools}
\usepackage{graphics,multimedia,tikz}

\newtheorem{thm}{Theorem}[section]
\newtheorem{cor}[thm]{Corollary}
\newtheorem{lm}[thm]{Lemma}

\newtheorem{clm}[thm]{Claim}
\newtheorem*{clm*}{Claim}
\theoremstyle{definition}

\newtheorem{exmp}[thm]{Example}
\newtheorem{remark}[thm]{Remark}

\numberwithin{equation}{section}
\newcommand{\sprf}{\noindent{\it Proof.}} 
\newcommand{\sqed}{\hfill\rule{1.3mm}{3mm}\medskip}

\newcommand{\cproof}{\noindent{\it Proof of claim.}\ } 
\newcommand{\cqed}{\hfill\rule{1.3mm}{3mm}}

\DeclareMathOperator{\id}{\textsl{id}}     

\DeclareMathOperator{\im}{im}         

\newcommand{\wec}[1]{{\mathbf{#1}}}  
\newcommand{\m}[1]{{\mathbf{\uppercase{#1}}}}

\newcommand{\bd}{\begin{description}}
\newcommand{\ed}{\end{description}}
\newcommand{\lb}{\langle} 

\newcommand{\cross}{{\sf Cross}}
\DeclareMathOperator{\Hom}{\mathcal{H}}
\DeclareMathOperator{\Sub}{\mathcal{S}}
\DeclareMathOperator{\Prod}{\mathcal{P}}

\newcommand{\ddd}{\underline{d}}

\begin{document}

\title[Cube term blockers]{Cube term blockers without finiteness}

\author{Keith A. Kearnes}
\address[Keith Kearnes]{Department of Mathematics\\
University of Colorado\\
Boulder, CO 80309-0395\\
USA}
\email{Keith.Kearnes@Colorado.EDU}

\author{\'Agnes Szendrei}
\address[\'Agnes Szendrei]{Department of Mathematics\\
University of Colorado\\
Boulder, CO 80309-0395\\
USA}
\email{Agnes.Szendrei@Colorado.EDU}

\thanks{This material is based upon work supported by
the National Science Foundation grant no.\ DMS 1500254 and
the Hungarian National Foundation for Scientific Research (OTKA)
grant no.\ K104251 and K115518.
}
\subjclass[2010]{Primary: 08B05; Secondary: 08A30, 08B20}
\keywords{cube term, cyclic term, Maltsev condition, idempotent variety,
  compatible cross}

\begin{abstract}
  We show that an idempotent variety
  has a $d$-dimensional cube term if
  and only if its free algebra on two generators has no
  $d$-ary compatible cross.
  We employ Hall's Marriage Theorem to show that a variety of
  finite signature whose fundamental operations have arities
  $n_1, \ldots, n_k$ has a
  $d$-dimensional cube term
  if and only if
  it has one of dimension $d=1+\sum_{i=1}^k (n_i-1)$.
  This lower bound on dimension is shown to be sharp.
  We show that a pure cyclic term variety has a cube term
  if and only if it contains no $2$-element semilattice.
We prove that 
the Maltsev condition ``existence of a cube term'' 
is join prime in the lattice of idempotent Maltsev conditions.
\end{abstract}

\maketitle

\setlength{\marginparwidth}{.8in}

\section{Introduction}\label{intro}
This note concerns a recently identified Maltsev condition,
which promises to be significant. It is called
``existence of a cube term''.

We begin by discussing relations.
A \emph{binary cross} on a set $A$ is a subset of $A^2$
of the form $(U_0\times A)\cup (A\times U_1)$,
where $U_0$ and $U_1$ are nonempty proper subsets of $A$.

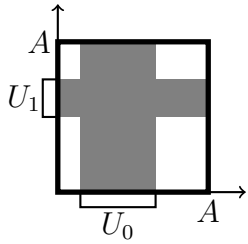
\begin{center}
\setlength{\unitlength}{1truemm}
\begin{picture}(30,37)
\put(-3,0){%
\begin{tikzpicture}[scale=1]

\fill[gray] (.3,0) rectangle(1.3,2);
\fill[gray] (0,1) rectangle(2,1.5);

\draw[->,line width=.8pt] (0,0) -- (2.5,0); 
\draw[->,line width=.8pt] (0,0) -- (0,2.5);
\node at (2,-0.25) {$A$};
\node at (-0.25,2) {$A$}; 
\draw[line width=2pt] (0,0) -- (2,0) -- 
                             (2,2) -- (0,2) -- cycle;
\draw[line width=.8pt] (.3,0) -- (.3,-0.2) -- 
                             (1.3,-0.2) -- (1.3,0);
\node at (0.8,-0.45) {$U_0$}; 
\draw[line width=.8pt] (0,1.5) -- (-0.2,1.5) -- 
                             (-0.2,1) -- (0,1);
\node at (-0.45,1.25) {$U_1$}; 

\end{tikzpicture}
}
\end{picture}
\end{center}

\bigskip

\noindent
If $U_0=U_1$, the cross is called \emph{symmetric}.
If $|U_0|=|U_1|=1$, the cross is called \emph{thin}.
The sequence $(U_0,U_1)$ is called the 
\emph{base (sequence)} for the cross.
If the cross is symmetric, i.e., if the base has the form 
$(U,U)$, then we also refer to $U$ as the base.

The definition of a cross makes sense for higher arity relations,
i.e. a $d$-ary cross is a subset of $A^d$ of the form 
\[
(U_0\times A\times\cdots \times A)\cup
(A\times U_1\times\cdots \times A)\cup
\cdots\cup
(A\times A\times\cdots \times U_{d-1})
\]
where $U_0,\dots,U_{d-1}$ are nonempty proper subsets of $A$.
For $d=1$ this means that a $1$-ary cross is a nonempty
proper subset of $A$. 
The definitions of symmetric cross, thin cross,
and base for a $d$-ary cross
are the expected ones. The arity of a cross
is also called its \emph{dimension}.

Our use in this paper of the notion of a cross
follows the earlier
use of crosses in the 1987 paper \cite{szendrei87},
which concerns the description of the maximal, locally closed
subclones of 
the clone of all idempotent operations on a given set.
In \cite{szendrei87}, symmetric and asymmetric \emph{thin} crosses
play a central role, although they are just called
`crosses'. In the current paper, we need
to consider arbitrary (`thick') crosses as well.

Now we turn to cube terms.
Let $\mathcal V$ be a variety and let $\m f = \m f_{\mathcal V}(x,y)$
be the $\mathcal V$-free algebra generated by the set $\{x,y\}$.
Since $\m f$ has an automorphism that switches $x$ and $y$, it follows
that exactly one of the following
two conditions holds: 
(i) the set $\{x,y\}^d-\{\wec{y}\}$ generates
$\wec{y}$, where $\wec{y}=(y,y,\ldots,y)$ is the constant
tuple with range $\{y\}$, or
(ii) 
different subsets of $\{x,y\}^d$ generate different
subalgebras of $\m f^d$.
For condition (i) to hold, $\mathcal V$
must have a term $c$ which applied to elements of 
$\{x,y\}^d-\{\wec{y}\}$ yields
$\wec{y}$, i.e., 
\begin{equation}\label{cube_id_vector}
\mathcal V\models c(\wec{z}_1,\wec{z}_2,\ldots)=\wec{y}
\quad\text{with all}\quad 
\wec{z}_i
\quad\text{in}\quad 
\{x,y\}^d-\{\wec{y}\}.
\end{equation}
This is a vector identity.
By considering this single vector identity coordinatewise,
this means that $\mathcal V$ satisfies $d$ identities of the form 
\begin{equation}\label{cube_identities}
c(\ldots, x, \ldots) = y
\end{equation}
where the only variables that appear in the identity are $x$ and $y$,
and for each place of $c$ there is an identity that has
$x$ in that place. For example, if $\mathcal V$ has a term $c$
satisfying 
\[
c\left(
\left[\begin{matrix} x\\y\end{matrix}\right],
\left[\begin{matrix} x\\x\end{matrix}\right],
\left[\begin{matrix} y\\x\end{matrix}\right]\right)=
\left[\begin{matrix} y\\y\end{matrix}\right],
\;\text{or equivalently, both of}\;
\ 
\left\{
\begin{matrix}
c(x,x,y)=y\hfill
\\
c(y,x,x)=y,\hfill
\end{matrix}
\right.
\]
then $c$ is a term of the desired type for $d=2$, which is called
a \emph{Maltsev term}.

A term $c$ satisfying the condition
described in \eqref{cube_id_vector} 
is called a
\emph{$d$-dimensional cube term} or just
\emph{$d$-cube term}
for $\mathcal V$.
Equivalently, $c$
is a \emph{$d$-cube term} if $d$ identities of the type in
(\ref{cube_identities}) suffice to establish the 
condition in (\ref{cube_identities}) for each place of $c$.
Clearly, a $d$-cube term for a variety $\mathcal V$
is automatically a $d'$-cube term for all $d'\ge d$.

Cube terms were introduced in 
\cite{berman-idziak-markovic-mckenzie-valeriote-willard}
as part of an investigation of finite algebras
with few subalgebras of powers.
Terms of equal strength, called 
\emph{parallelogram terms}, were discovered
independently and at the same time 
in the study of finitely related clones, 
\cite{kearnes-szendrei2012}.
Cube terms and their equivalents have played roles
in 
\cite{idziak-markovic-mckenzie-valeriote-willard}
in the study of constraint satisfaction problems, in
\cite{aichinger-mayr-mckenzie,barto,kearnes-szendrei2012}
in the study of finitely related clones,
in 
\cite{kearnes-szendrei,moore} in natural duality theory,
and in
\cite{bulatov-mayr-szendrei} concerning the subpower membership problem.

Theorem 2.1 of \cite{markovic-maroti-mckenzie}
gives a method 
for recognizing if a finite idempotent
algebra has no cube term. Namely, a finite
idempotent algebra fails to have a $d$-cube term for any $d$ 
if and only
if it has a \emph{cube term blocker}.
A {cube term blocker} of a finite
idempotent algebra $\m a$ is defined to be a pair
$(U,B)$ of nonempty
subuniverses of $\m a$, with $U\subsetneq B$,
such that $U$ serves as a base for a compatible, 
symmetric, $d$-ary cross of $\m b$
for every $d$. It follows that a finite idempotent algebra
$\m a$ fails to have a $d$-cube
term for any $d$
if and only if some subalgebra $\m b\leq \m a$ has 
compatible symmetric crosses of every arity.

The result of \cite{markovic-maroti-mckenzie}
does not help if one wants to show that $\m a$ has no
$d$-cube term for a fixed $d$. The result
also does not help if $\m a$ is infinite. But 
Lemma 2.8 of \cite{kearnes-tschantz}
shows that an idempotent variety $\mathcal V$ fails to have a 
Maltsev term (i.e. a $2$-cube term)
if and only if the free $\mathcal V$-algebra
on $2$ generators has a compatible $2$-ary cross.
Here $\mathcal V$ need not satisfy any finiteness
hypothesis, but the cross involved
turns out to be asymmetric, while the notion
of a cube term blocker involves symmetric crosses only. 
Furthermore, Lemma 2.8 of \cite{kearnes-tschantz}
is a result about $2$-cube terms only.

The current paper may be viewed as establishing a
generalization of both Theorem~2.1 of \cite{markovic-maroti-mckenzie}
and Lemma~2.8 of \cite{kearnes-tschantz}.
We will prove that an
idempotent variety $\mathcal V$ fails to have a 
$d$-cube term
if and only if the free $\mathcal V$-algebra
on $2$ generators, $\m f=\m f_{\mathcal V}(x,y)$,  
has a compatible $d$-ary cross.
We will also show that an
idempotent variety $\mathcal V$ fails to have a 
$d$-cube term for every $d$
if and only if $\m f$
has a nonempty proper subuniverse $U$ that serves as a base for
symmetric crosses of all arities, i.e. $(U,F)$ is a cube
term blocker of $\m f$.
The message to take away from this is that to avoid
cube terms of a fixed dimension one should work with a
not-necessarily-symmetric cross of that dimension, but to avoid cube terms
of all dimensions it suffices to work with symmetric crosses
or cube term blockers.

This note evolved in response to a question
we learned from Cliff Bergman, which we state and
answer in Section~\ref{bergman-section}.
Section~\ref{producing}
contains our proof that $\m f=\m f_{\mathcal V}(x,y)$
has compatible crosses of all arities if and only if
$\m f$ has a subuniverse $U$ such that
$(U,F)$ is a cube term blocker.
In Section~\ref{example}
we develop a tight lower bound
on the minimal dimension of a cube
term for idempotent varieties
of finite signature. 
In Section~\ref{generic-section} we use our results to establish that
the Maltsev condition ``existence of a cube term'' 
is join prime in the lattice of idempotent Maltsev conditions.

\section{Cube terms and crosses}

A nonempty subset $B$ of a product
$A_0\times \cdots \times A_{d-1}$  is a
\emph{block} if it is a product subset:
$B = B_0\times \cdots \times B_{d-1}$ with $B_i\subseteq A_i$
for all $i$. A block is \emph{full} in the $i$-th coordinate 
if $B_i=A_i$. Thus, for any sequence $(U_0,\dots,U_{d-1})$
of nonempty proper subsets of $A$, the cross 
\begin{align*}
\cross(U_0,\ldots,U_{d-1})&{}=
(U_0\times A\times \cdots \times A)\cup
\cdots\cup
(A\times\cdots \times A\times U_{d-1})\\
&{}=B_0\cup \cdots \cup B_{d-1}
\end{align*}
on $A$
is defined to be a subset of $A^d$ that is
a union of $d$ blocks $B_0,\dots,B_{d-1}$ 
where $B_i$ is full in all coordinates except the $i$-th.

If $t$ is an operation on a set $A$ and
$U\subseteq A$, then $t$ is 
\emph{$U$-absorbing in its $i$-th variable} if
\[
t(A,A,\ldots,\underbrace{U}_i,\ldots,A)\subseteq U.
\]
An operation $t$ on a set $A$ is \emph{idempotent}
if $t(a,a,\ldots,a)=a$ for all $a\in A$.

\begin{lm}\label{lm-basics}
Let $A$ be a set with nonempty proper subsets 
$U$, $U_0, \ldots, U_{d-1}$, and let $t$ be an $n$-ary
idempotent operation on $A$.
\begin{enumerate}

\item[(1)]
If $t$ is compatible with 
$\cross(U_0,\ldots,U_{d-1})$ and $\pi\in S_d$
is a permutation, then 
$t$ is compatible with 
$\cross(U_{\pi(0)},\ldots,U_{\pi(d-1)})$.
If $(i_0,\ldots,i_{e-1})$ is a subsequence of 
$(0,\ldots,d-1)$, then 
$t$ is also compatible with 
$\cross(U_{i_0},\ldots,U_{i_{e-1}})$.

\item[(2)] 
If $t$ is compatible with 
$\cross(U_0,\ldots,U_{d-1})$, then each $U_i$
is a subuniverse of $(A,t)$.

\item[(3)]
$t$ 
is compatible with $\cross(U_0,\ldots,U_{d-1})$ 
if and only if 
\begin{quote}
$(\ast)$ for every function
\[
m\colon \{0,\ldots,{n-1}\}\to \{0,\ldots,{d-1}\}
\]
there is some $i\in \im(m)$ such that 
\begin{equation}\label{cross-preservation}
\text{$a_j\in U_{i}$ for all $j\in m^{-1}(i)$}
\quad\Longrightarrow\quad
t(a_0,\ldots,a_{n-1})\in U_{i}.
\end{equation}
\end{quote}

\item[(4)]
If $t$ is compatible with 
$\cross(U_0,\ldots,U_{d-1})$, 
and $n\leq d$,
then for all except possibly 
$n-1$ choices of $j<d$
it is the case that $t$ is 
$U_j$-absorbing in one of its variables.

\item[(5)] 
The following are equivalent for
$t$:
\begin{enumerate}
\item[(i)]
$t$ is compatible with the $d$-ary symmetric cross
$\cross(U,\ldots,U)$ for some $d\geq n$.
\item[(ii)]
$t$ is $U$-absorbing in one of its variables.
\item[(iii)]
$t$ is compatible with the $d$-ary symmetric cross
$\cross(U,\ldots,U)$ for every $d\geq 1$.
\end{enumerate}
\end{enumerate}
\end{lm}

\begin{proof}
For the first statement in (1) observe that
$\cross(U_{\pi(0)},\ldots,U_{\pi(d-1)})$
differs from
$\cross(U_0,\ldots,U_{d-1})$ 
by a permution of coordinates.
Therefore the desired conclusion follows from the fact that 
if we permute the coordinates of a subuniverse of $(A;t)^d$
we again get a subuniverse of $(A;t)^d$.

For the second statement in (1), choose $a_0\notin U_0$. 
Since $t$ is idempotent, the singleton
$\{a_0\}$ is a subuniverse of $(A;t)$.
Hence
\[
\{a_0\}\times\cross(U_1,\ldots,U_{d-1})
=\cross(U_0,\ldots,U_{d-1})\cap(\{a_0\}\times A^{d-1})
\]
is a subuniverse of $(A;t)^d$.
It follows that $\cross(U_1,\ldots,U_{d-1})$ is a subuniverse of
$(A;t)^{d-1}$.
Similarly, $\cross(U_0,\ldots,U_{i-1},U_{i+1},\dots,U_{d-1})$ is a subuniverse of
$(A;t)^{d-1}$ for all $i<d$.
Repeating this procedure for every $i$ not occurring in $(i_0,\dots,i_{e-1})$
we get that 
$\cross(U_{i_0},\ldots,U_{i_{e-1}})$
is a subuniverse of $(A;t)^e$, that is, 
$t$ is compatible with $\cross(U_{i_0},\ldots,U_{i_{e-1}})$.

Item (2) is the special case $e=1$ of the second statement in (1).

For item (3), assume first 
that $(\ast)$ fails.
Then there is a function
\[
m\colon \{0,\ldots,{n-1}\}\to \{0,\ldots,{d-1}\}
\]
such that for every $i\in \im(m)$ the implication in
\eqref{cross-preservation} fails.
Choose witnessing elements: i.e.
choose, for each $i\in \im(m)=\{i_0,\ldots,i_{e-1}\}$, elements
$a_{i,j}\in A$ ($j<n$) satisfying $a_{i,j}\in U_i$ for all
$j\in m^{-1}(i)$ such that 
$t(a_{i,0},\ldots,a_{i,n-1})\notin U_{i}$.
The columns of the matrix $[a_{k,\ell}]$ lie
in $\cross(U_{i_0},\ldots,U_{i_{e-1}})$, 
because every $j<n$ belongs to $m^{-1}(i)$ for some
$i\in \im(m)=\{i_0,\ldots,i_{e-1}\}$, and hence by construction,
the $j$-th column of $[a_{k,\ell}]$
has $i$-th entry $a_{i,j}\in U_i$. 
However, by construction, the column
obtained by applying $t$ to the rows of this matrix
does not lie in $\cross(U_{i_0},\ldots,U_{i_{e-1}})$.
Thus, $t$ is not compatible with 
$\cross(U_{i_0},\ldots,U_{i_{e-1}})$, so
it is not compatible with 
$\cross(U_0,\ldots,U_{d-1})$ either, according to item (1).

Conversely, 
assume that $t$ is not compatible with 
$\cross(U_0,\ldots,U_{d-1})$. Then we can select elements from
this relation and allow them to serve as columns for a $d\times n$ matrix
$[a_{i,j}]$ where (i) for each column $j$, there is a row index $i\;(=:m(j))$
such that $a_{i,j}\in U_i$ and (ii) the value obtained
by applyig $t$ to the $i$-th row fails to belong to $U_i$
for every $i$.
This yields a function $m$ witnessing that $(\ast)$ fails.

For item (4), define a bipartite graph with one part
$X=\{x_0,\ldots,x_{n-1}\}$, other part equal to $\ddd:=\{0,\dots,d-1\}$,
and edge relation containing
exactly those pairs $(x_i, j)$ where $t$ is not $U_j$-absorbing
in variable $x_i$.

\begin{clm}\label{no-matching}
  There is no matching from $X$ to $\ddd$.
  (A {\it matching} from $X$ to $\ddd$
  is a subset of the edge set that is an injective function
  $X\to\ddd$.)
\end{clm}

\begin{proof}[Proof of Claim~\ref{no-matching}.]
Assume that there is a matching 
$\mu\colon X\to \ddd$.
By item~(1) we may
reorder the $U_j$'s so that
$\mu(x_i)=i$ is the matching.
It follows that for every $i<n$, 
$t$ is not $U_i$-absorbing in its $i$-th variable,
so there must exist $u_{i,i}\in U_{i}$ and 
elements $a_{i,k}\in A$ such that 
\[
t(a_{i,0},a_{i,1},\ldots,u_{i,i},\ldots,a_{i,n-1})\notin U_{i}.
\]
There also exist $a_j\notin U_j$ for every $j<d$.
These ingredients yield a situation
\begin{equation}\label{cross_not_preserved}
t\left(
\left[
\begin{matrix}
u_{0,0} \\
a_{1,0} \\
\vdots \\
a_{n-1,0} \\
\hline
a_n\\
\vdots\\
a_{d-1}
\end{matrix}
\right], 
\left[
\begin{matrix}
a_{0,1} \\
u_{1,1} \\
\vdots \\
a_{n-1,1} \\
\hline
a_n\\
\vdots\\
a_{d-1}
\end{matrix}
\right], 
\cdots,
\left[
\begin{matrix}
a_{0,n-1} \\
a_{1,n-1} \\
\vdots \\
u_{n-1,n-1} \\
\hline
a_n\\
\vdots\\
a_{d-1}
\end{matrix}
\right]
\right)=
\left[
\begin{matrix}
\notin U_{0}\\
\notin U_{1}\\
\vdots\\
\notin U_{n-1}\\
\hline
a_n(\notin U_{n})\\
\vdots\\
a_{d-1}(\notin U_{d-1})
\end{matrix}
\right].
\end{equation}
The operands are in $\cross(U_0,\ldots,U_{d-1})$,
but the value is not, a contradiction.
\renewcommand{\qedsymbol}{$\diamond$} 
\end{proof}

By the Marriage Theorem, there is a subset $Y\subseteq X$
such that the set $K\subseteq \ddd$
of elements adjacent to elements of $Y$ satisfies 
$|K|<|Y|\leq n$. Thus, $Y\not=\emptyset$ and 
the set $K$
has size at most $n-1$; moreover, if $j\in \ddd-K$
then no element of $Y$ is adjacent to $j$. Hence 
$t$ is $U_j$-absorbing in its variables in $Y$.

For item (5), the implication (iii)~$\Rightarrow$~(i) is a tautology,
and the implcation (i)~$\Rightarrow$~(ii) follows from
the statement in (4) we just proved.

To establish the remaining implication (ii)~$\Rightarrow$~(iii) assume
without loss of generality that $t$ is $U$-absorbing in its first variable,
and consider the $d$-ary symmetric cross $\cross(U,\dots,U)$ for some $d\ge1$.
Let $[a_{i,j}]$ be a $d\times n$ matrix of element of $A$ whose columns 
are in $\cross(U,\dots,U)$. In particular, the first column lies in
$\cross(U,\dots,U)$, so $a_{i,0}\in U$ for some $i<d$.
Since $t$ is $U$-absorbing in its first variable, we
get that $t(a_{i,0},\dots,a_{i,n-1})\in U$. Hence the column
obtained by applying $t$ to the rows of the matrix $[a_{i,j}]$ lies in
$\cross(U,\dots,U)$. This proves that $t$ is compatible with 
$\cross(U,\dots,U)$.
\end{proof}

\begin{cor}\label{comp_crosses_of_cube_terms}
  Let $\m a$ be an idempotent algebra.
If\/ $\m a$ has a $d$-cube term, then $\m a$ has no compatible
cross of dimension $d$ or larger.
\end{cor}

\begin{proof}
Let $c$ be a $d$-cube term for $\m a$, and assume
$\m a$ has an $e$-dimensional
compatible cross
$\cross(U_0,\dots,U_{e-1})$ with $d\le e$.
By Lemma~\ref{lm-basics}~(4) there exists
$j<d$
such that 
$c$ is $U_j$-absorbing in one of its variables.
This contradicts the cube identities (see 
\eqref{cube_id_vector} or \eqref{cube_identities}).
\end{proof}

Our goal in this section is to characterize
idempotent varieties which have no $d$-cube term
(for a fixed $d\ge2$) or have no cube term (of any dimension).
In Theorem~\ref{main1} below
the first property will be characterized by the existence of a $d$-dimensional
cross, while the second one will be characterized
by the existence of a specific infinite system of crosses, which we call
a `cross sequence'.
A \emph{cross sequence} for an idempotent algebra
$\m a$ is an $\omega$-sequence $(U_0,U_1,\ldots)$ of 
proper nonempty subsets of $A$
such that $\cross(U_{i_0},\dots,U_{i_{k-1}})$ is a compatible
relation of $\m a$ for every
finite subsequence $(U_{i_0},\dots,U_{i_{k-1}})$ of 
$(U_0,U_1,\ldots)$.
A cross sequence $(U_0,U_1,\ldots)$
is \emph{proper} if $\bigcap_{i<\omega} U_i\neq \emptyset$ and 
$\bigcup_{i<\omega} U_i\neq A$.
It follows from Lemma~\ref{lm-basics}(1)
that any subsequence and any reordering 
of a (proper) cross sequence
is again a (proper) cross sequence.

\begin{thm}\label{main1}
Let $\mathcal V$ be an idempotent variety, let
$\m f = \m f_{\mathcal V}(x,y)$ be the 
$\mathcal V$-free algebra over the free generating set $\{x,y\}$, and let
$d$ be a positive integer. 
\begin{enumerate}
\item[(1)]
$\mathcal V$ has no $d$-cube term iff\/
$\m f$ has a compatible $d$-dimensional cross.
\item[(2)]
$\mathcal V$ has no cube term of any dimension iff\/
$\m f$ has a proper cross sequence iff\/ 
$\m f$ has a cross sequence.
\end{enumerate}
\end{thm}

We will prove the two statements of Theorem~\ref{main1}
simultaneously.
In Theorem~\ref{thm_cross-seqs}
a uniform formulation of these two statements
is based on the observation that
for any variety $\mathcal V$, the condition `$\mathcal V$
has no $d$-cube term' for a fixed $d$
is equivalent to
\begin{center}
  $\mathcal V$ has no $e$-cube term for any $e<\delta$ 
\end{center}
for $\delta=d+1$,
while the condition
`$\mathcal V$ has no cube term of any dimension' 
is equivalent to the same displayed condition for $\delta=\omega$.

For $0<\delta\le\omega$, let
\[
\delta^-:=\begin{cases}
          \delta-1 & \text{if $\delta<\omega$},\\
          \omega & \text{if $\delta=\omega$}.
         \end{cases}
\]

\begin{thm}\label{thm_cross-seqs}
Let $\mathcal V$ be an idempotent variety,
let $\m f = \m f_{\mathcal V}(x,y)$
be the free $\mathcal V$-algebra generated by $\{x,y\}$,
and let $2\le\delta\le\omega$.
The variety
$\mathcal V$ fails to have a $d$-cube for any $d<\delta$
if and only if there is a
$\delta^-$-sequence, $\sigma=(U_0,U_1,\ldots)=(U_i)_{i<\delta^-}$, of 
subuniverses of $\m f$ such that
\begin{enumerate}
\item[(1)] $x\in U_i$ and $y\notin U_i$ for every $i<\delta^-$, and
\item[(2)]
$\cross(U_{i_0},\ldots,U_{i_{k-1}})$ is a compatible relation of 
$\m f$ for every finite subsequence $(U_{i_0},\ldots,U_{i_{k-1}})$
of $\sigma$.
\end{enumerate}
\end{thm}

\begin{proof}
The ``if'' assertion follows from Corollary~\ref{comp_crosses_of_cube_terms}:
if $\m f$ has compatible crosses of every arity $<\delta$,
then it cannot have a $d$-cube term for any  $d<\delta$.

For the converse, assume that $\mathcal V$ has no $d$-cube term
for any $d<\delta$.
This implies, in particular, that $\mathcal V$ is nontrivial.

Recursively define the sequence
$\sigma = (U_0,U_1,\ldots)=(U_i)_{i<\delta^-}$ 
with $U_i\leq \m f$, according to the
following rule: for $i<\delta^-$, $U_i$ is chosen
to be a subuniverse maximal for the properties that
\begin{enumerate}
\item[(i)] $x\in U_i$, and
\item[(ii)] 
$\wec{y}\notin \langle 
\cross(U_0,U_1,\ldots,U_i,\{x\},\ldots,\{x\})\rangle_{\m f^k}$
for any $k$ with $i< k<\delta$.
\end{enumerate}
It is possible to make these choices, for the following reasons.
The fact that $\mathcal V$ does not have a $d$-cube term for any $d<\delta$
means exactly
that 
$\wec{y}\notin \langle 
\cross(\{x\},\{x\},\ldots,\{x\})\rangle_{\m f^k}$
for every $k<\delta$. Thus the subuniverse
$\{x\}$ satisfies all the properties
required of $U_0$, except that it need not be maximal
among the subuniverses satisfying (i) and (ii) for $i=0$.
But the union of a chain
of subuniverses of $\m F$ satisfying (i) and (ii) 
for a given $i$ again satisfies
these conditions for that $i$, so $\{x\}$ can be extended
to a maximal subuniverse $U_0$ satisfying (i) and (ii) for $i=0$.
Similarly, if $0<i<\delta^-$ and
$U_0, \ldots, U_{i-1}$ have been chosen, 
then condition (ii) for $i-1$
guarantees that
$\{x\}$ satisfies all the properties required of $U_i$
except maximality. Extend $\{x\}$ to a maximal $U_i$, etc.

Observe that $y\notin U_i$ for any $i<\delta^-$, since otherwise 
$\{x,y\}\subseteq U_i$, leading to $F=U_i$, leading
to $\cross(U_0,\ldots,U_i)=F^{i+1}$,
contradicting item (ii) above. Hence item (1)
of the theorem statement holds for 
$\sigma=(U_0, U_1, \ldots)=(U_i)_{i<\delta^-}$.

Our remaining task is to show that 
$\cross(U_{i_0},\ldots,U_{i_{k-1}})$ is a compatible relation of 
$\m f$ for every finite subsequence $(U_{i_0},\ldots,U_{i_{k-1}})$
of $\sigma$.
Every finite subsequence of $\sigma$ is a subsequence of an initial segment
of $\sigma$, therefore, in view of Lemma~\ref{lm-basics}~(1), 
it suffices to prove that 
$\cross(U_0,\ldots,U_{d-1})$ is a compatible relation
of $\m f$ for every $d<\delta$. 
We will do this
simultaneously for every $d$ with an indirect argument.
Assume that
there is some $d<\delta$ and some element
\begin{equation}\label{d_and_p}
\wec{p} = (p_0,\ldots,p_{d-1})\in
\langle \cross(U_0,\ldots,U_{d-1})\rangle_{{\m f}^d}
- \cross(U_0,\ldots,U_{d-1}).
\end{equation}
Necessarily $p_i\notin U_i$ for any $i<d$.
There is no harm in assuming that, among all possible
choices of $d$ and $\wec{p}$, we have chosen those
such that $p_i\neq y$ holds in the fewest number of coordinates.
That is, we assume that \eqref{d_and_p} holds
and also that for no $e<\delta$ do we have 
$\wec{q}\in\langle
\cross(U_0,\ldots,U_{e-1})\rangle_{{\m f}^e} - \cross(U_0,\ldots,U_{e-1})$ 
with $\wec{q}$ 
differing from $\wec{y}$ in strictly
fewer coordinates than $\wec{p}$.

Since $\wec{p}\in\langle\cross(U_0,\ldots,U_{d-1})\rangle_{{\m f}^d}$, there
exists a term $s$ and there exist elements 
$\wec{b}_0,\ldots,\wec{b}_{m-1}\in \cross(U_0,\ldots,U_{d-1})$
such that 
$
s(\wec{b}_0, \ldots, \wec{b}_{m-1})=\wec{p}.
$ 
Observe that one may lengthen all tuples
involved by adding some number 
of $y$'s (say $g$ with $d+g<\delta$) to the end in order to obtain
\begin{equation}\label{lengthen}
s\left(
\left[
\begin{matrix}
\wec{b}_0\\
y\\
\vdots\\
y\\
\end{matrix}\right], 
\ldots, 
\left[
\begin{matrix}
\wec{b}_{m-1}\\
y\\
\vdots\\
y\\
\end{matrix}\right]
\right)=
\left[\begin{matrix}
\wec{p}\\
y\\
\vdots\\
y\\
\end{matrix}\right].
\end{equation}
The columns appearing as arguments to $s$ in (\ref{lengthen})
belong to 
\[
\cross(U_0,\ldots,U_{d-1},V_d,\ldots,V_{d+g-1})
\]
for any choice of
 $(\emptyset\not=)\;V_i\subset F$.

There must exist some coordinate of $\wec{p}$
that is not $y$, else condition (ii) from the definition
of $\sigma$
is violated
when $k=d$. Let $i$ be the first coordinate of $\wec{p}$
where $p_i\neq y$; hence $i< d$.
Since $p_i\notin U_i$,
the subuniverse $U_i' = \langle U_i\cup \{p_i\}\rangle_{\m f}$
properly extends $U_i$. By the maximality of $U_i$,
there must exist some $k$ with $i< k<\delta$ such that
$\wec{y}\in \langle 
\cross(U_0,\ldots,U_i',\{x\},\ldots,\{x\})\rangle_{\m f^k}$.

To understand how $\wec{y}$ could be generated, observe that
since
\begin{multline*}
\cross(U_0,\ldots,U_i',\{x\},\ldots,\{x\})\\
= 
B_0(U_0)\cup \dots \cup B_{i-1}(U_{i-1})\cup B_i(U_i')\cup B_{i+1}(\{x\})\cup
\dots\cup B_{k-1}(\{x\}),
\end{multline*}
we get that 
$\langle\cross(U_0,\dots,U_i',\{x\},\ldots,\{x\})\rangle_{{\m f}^k}$
equals the subalgebra join
\[
B_0(U_0)\vee\dots \vee B_{i-1}(U_{i-1})\vee B_i(U_i')\vee
B_{i+1}(\{x\})
\vee
\dots\vee 
B_{k-1}(\{x\}).
\]
However $B_i(U_i')=F^i\times U_i'\times F^{k-i-1}$ is generated by 
$\{x,y\}^i\times (U_i\cup\{p_i\})\times \{x,y\}^{k-i-1}$, and
all elements of this product set already belong to 
$\cross(U_0,\ldots,U_i,\{x\},\ldots,\{x\})$ except the tuple
$(y,\ldots,y,p_i,y,\ldots,y)$. Hence
$\langle\cross(U_0,\ldots,U_i',\{x\},\ldots,\{x\})\rangle_{{\m f}^k}$
is generated by 
\[
\cross(U_0,\ldots,U_i,\{x\},\ldots,\{x\})
\cup
\{(y,\ldots,y,p_i,y,\ldots,y)\}.
\]
Since $\wec{y}$ is generated by this set, there is a term
$t$ and elements (columns) $\wec{c}_i\in
\cross(U_0,\ldots,U_i,\{x\},\ldots,\{x\})\subseteq F^k$ such that
\begin{equation}\label{generate}
t\left(
\wec{c}_0, \ldots, \wec{c}_{n-1},
\left[
\begin{matrix}
y \\
\vdots \\
y \\
p_{i} \\
y \\
\vdots\\
y
\end{matrix}
\right]
\right)=
\left[
\begin{matrix}
y\\
\vdots\\
y\\
y\\
y\\
\vdots\\
y
\end{matrix}
\right]=\wec{y}.
\end{equation}
Lengthen these tuples, by adding $y$'s at the end,
to some length $e$ satisfying $\max\{k,d\}\le e<\delta$
(hence the columns have length at least $d$). 
Equation \eqref{generate} still
holds. Write the extension of $\wec{c}_i$ as $\wec{c}_i'$,
and note that 
\[
\wec{c}_i'\in \cross(U_0,\ldots,U_i,\{x\},\ldots,\{x\})\ 
(\subseteq F^e)
\] 
where there
may be 
more $\{x\}$'s than before.

Let $\varepsilon = (\varepsilon_0,\ldots,\varepsilon_{e-1})$ 
be the endomorphism of $\m f^e$ defined coordinatewise as follows:
\begin{enumerate}
\item[(a)] $\varepsilon_j\colon \m f\to \m f$ is the identity if
$0\leq j\leq i\ (\le d)$ or $d\le j<e$, and
\item[(b)]
$\varepsilon_j\colon \m f\to \m f\colon x\mapsto x, y\mapsto p_j$ 
if $i<j< d$.
\end{enumerate}

Observe that $\varepsilon$ maps
$\cross(U_0,\ldots,U_i,\{x\},\ldots,\{x\})$ into itself.
Indeed, $\varepsilon(B_j(U_j))\subseteq B_j(U_j)$ for 
$j\leq i$ because $\varepsilon_j=\id$, while
$\varepsilon(B_j(\{x\}))\subseteq B_j(\{x\})$ for 
$j>i$ because $\varepsilon_j$ fixes $x$.

Applying $\varepsilon$ to (\ref{generate}) yields
\begin{equation}\label{q}
t\left(
\varepsilon(\wec{c}_0'), \ldots, 
\varepsilon(\wec{c}_{n-1}'),
\left[
\begin{array}{c}
y \\
\vdots \\
y \\
p_{i} \\
p_{i+1} \\
\vdots\\
p_{d-1}\\
y\\
\vdots\\
y\\
\end{array}
\right]
\right)=
\left[
\begin{array}{c}
y\\
\vdots\\
y\\
y\\
p_{i+1} \\
\vdots\\
p_{d-1}\\
y\\
\vdots\\
y\\
\end{array}
\right]=:\wec{q}.
\end{equation}
By the choice of $i$, the last argument of $t$ in \eqref{q}
is the column $(\wec{p},y,\dots,y)$. Therefore, the expression for $\wec{q}$
in \eqref{q} can be rewritten as
\begin{equation}\label{q2}
t\left(
\varepsilon(\wec{c}_0'), \ldots, 
\varepsilon(\wec{c}_{n-1}'),
s\left(
\left[
\begin{array}{c}
\wec{b}_0\\
y\\
\vdots\\
y\\
\end{array}\right], 
\ldots, 
\left[
\begin{array}{c}
\wec{b}_{m-1}\\
y\\
\vdots\\
y\\
\end{array}\right]
\right)\right)=\wec{q}
\end{equation}
using (\ref{lengthen}).
The columns, $\varepsilon(\wec{c}_u')$ and $(\wec{b}_v,y,\ldots,y)$
all belong to $\cross(U_0,\ldots,U_{e-1})$,
but $q_j\notin U_j$ for any $j$.
Hence (\ref{q2}) asserts that
\[
\wec{q}\in 
\langle\cross(U_0,\ldots,U_{e-1})\rangle_{{\m f}^e}
-\cross(U_0,\ldots,U_{e-1})
\]
with $\wec{q}$ 
differing from the constant $y$-tuple in strictly
fewer coordinates than $\wec{p}$. 
This is so because $\wec{q}$
differs from $\wec{p}$ only in that it may have 
more $y$'s at the end and $q_i=y$ while $p_i\neq y$.
This conclusion contradicts
the choice of $\wec{p}$. This proves that
$\cross(U_0,\ldots,U_{d-1})$ is a compatible relation of $\m f$
for every $d<\delta$.
\end{proof}

\begin{proof}[Proof of Theorem~\ref{main1}]
For item (1), we
assume first that $\mathcal V$ has no $d$-cube term.
Hence, $\mathcal V$ has no $e$-cube term for any $e<d+1$.
It follows from Theorem~\ref{thm_cross-seqs} (for $\delta=d+1$) 
that there is a
sequence $(U_0,\ldots,U_{d-1})$ of 
subuniverses of $\m f$ such that
$x\in \bigcap_{i<d}U_i$, $y\notin\bigcup_{i<d} U_i$, and
$\cross(U_0,\ldots,U_{d-1})$ is a compatible relation of 
$\m f$. 

Conversely, 
if $\m f$ has a compatible $d$-dimensional cross,
then, by Corollary~\ref{comp_crosses_of_cube_terms},
$\m f$ has no $d$-cube term. Hence $\mathcal V$ has no $d$-cube term.

For item (2), let us assume first that $\mathcal V$ has no cube term.
Applying Theorem~\ref{thm_cross-seqs} (for $\delta=\omega$)
we conclude that 
there is an
$\omega$-sequence, $\sigma=(U_0,U_1,\dots)$, of 
subuniverses of $\m f$ such that
$\sigma$ is a proper cross sequence for $\m f$
satisfying
$x\in \bigcap_{i<\omega}U_i$ and $y\notin\bigcup_{i<\omega} U_i$.

If $\m f$ has a proper cross sequence, then $\m f$ has a cross sequence.
Finally, if $\m f$ has a cross sequence, then $\m f$ has compatible crosses
of arbitrarily high dimensions. Therefore, 
by Corollary~\ref{comp_crosses_of_cube_terms},
$\m f$ has no cube term of any dimension. 
Hence $\mathcal V$ has no cube term of any dimension either,
completing the proof.
\end{proof}

\section{Producing symmetric crosses}\label{producing}

We have shown in Theorem~\ref{main1}~(2) 
that an idempotent variety $\mathcal V$ fails to have a cube term if
and only if its $2$-generated free algebra $\m f$ has a 
proper cross sequence. 
In this section we will 
use a combinatorial argument to
show that this cross sequence can be converted to
a constant cross sequence $(U,U,U,\ldots)$. This produces 
a nonempty proper subuniverse $U$ of $\m f$ that is
the base of a 
compatible 
\emph{symmetric}
$d$-ary cross for every $d$.
In fact, our construction of `symmetrizing' a proper cross sequence works for
any idempotent algebra, as the theorem below shows.

\begin{thm}\label{main2}
The following are equivalent for an idempotent algebra
$\m a$.
\begin{enumerate}
\item[(1)] $\m a$ has a proper cross sequence.
\item[(2)] $\m a$ has a nonempty proper subuniverse $U$ such that
$(U,A)$ is a cube term blocker for $\m a$.
(That is, $U$ is a base for compatible symmetric $d$-ary crosses
of $\m a$ for all $d$.)
\end{enumerate}
\end{thm}

\begin{proof}
The implication $(2)\Rightarrow(1)$ is clear from the definitions: 
If $(U,A)$ is a cube term blocker of $\m a$,
then the $d$-ary cross $\cross(U,\dots,U)$
is a compatible relation of $\m a$ for every $d$, so
the constant $\omega$-sequence $(U,U,\ldots)$ is a proper
cross sequence for $\m a$.

To prove the reverse implication $(1)\Rightarrow(2)$,
assume that $\sigma=(U_0,U_1,\ldots)$ is a proper cross sequence
for $\m a$. We shall prove that if 
$U:=\bigcup_{i=0}^{\infty} \left(\bigcap_{j=i}^{\infty} U_j\right)$,
then $(U,F)$ is a cube term blocker of $\m a$.
Note that 
\[
\bigcap_{j=0}^{\infty} U_j\subseteq
\bigcap_{j=1}^{\infty} U_j\subseteq
\dots\subseteq
\bigcap_{j=i}^{\infty} U_j\subseteq
\bigcap_{j=i+1}^{\infty} U_j\subseteq
\dots
\subseteq
\bigcup_{j=0}^{\infty} U_j,
\]
so $U:=\bigcup_{i=0}^{\infty} \left(\bigcap_{j=i}^{\infty} U_j\right)$
is the union of an increasing $\omega$-sequence
of subuniverses of $\m a$. It follows that $U$ is a subuniverse of $\m a$.
Moreover,
since the cross sequence $\sigma=(U_0,U_1,\dots)$
is proper, we have
$\bigcap_{j=0}^{\infty} U_j\not=\emptyset$ and 
$\bigcup_{j=0}^{\infty} U_j\not= A$. This implies that
$U$ is a nonempty proper subuniverse of $\m a$.

The nontrivial part of the proof, therefore, 
is the argument that $U$ is a base for symmetric
crosses of all arities. To see that this is so,
choose an arbitrary term operation $t(x_0,\ldots,x_{n-1})$ of $\m a$
and as in the proof of Lemma~\ref{lm-basics}~(4),
use it to define a bipartite graph as follows:
the vertices of the two parts are 
$X = \{x_0,\ldots,x_{n-1}\}$,
the set of variables of $t$, and $\omega$, 
the set indexing the terms of the cross sequence
$\sigma=(U_0,U_1,\dots)$;
the edges of the graph are the pairs
$(x_i,j)$ such that $t$ is not
$U_j$-absorbing in variable $x_i$.

\begin{clm}\label{no-matching2}
There is no matching from $X$ to $\omega$.
\end{clm}

\begin{proof}[Proof of Claim~\ref{no-matching2}.]
Suppose that there is a matching 
$\mu\colon X\to \omega$.
Since a reordering of finitely many terms of $\sigma$
produces a new cross sequence for $\m a$ and
leaves the set $U$ unchanged, we 
may assume without loss of generality that
$\mu(x_i)=i$ for all $i<n$. 
This mean that for every $i<n$ the term $t$ is not
$U_i$-absorbing in variable $x_i$, so 
for every $i<n$ there must exist $u_{i,i}\in U_{i}$ and 
elements $a_{i,k}\in A$ such that 
\[
t(a_{i,0},a_{i,1},\ldots,u_{i,i},\ldots,a_{i,n-1})\notin U_{i}.
\]
There also exist $a_j\notin U_j$ for every $j<d$.
Now the same calculation \eqref{cross_not_preserved} 
as in the proof of Lemma~\ref{lm-basics}~(4)
yields that $t$ is not compatible with $\cross(U_0,\dots,U_{d-1})$.
This contradicts our assumption that $(U_0,U_1,\dots)$ is a cross
sequence for $\m a$, and hence proves the claim.
\renewcommand{\qedsymbol}{$\diamond$} 
\end{proof}

Since $X$ is finite, the Marriage Theorem holds in this situation.
It asserts that, since there is no matching from $X$ to $\omega$,
there is a subset $Y\subseteq X$
such that the set $K\subseteq \omega$
of elements adjacent to elements of $Y$ satisfies 
$|K|<|Y|\le |X|= n$. Therefore $Y\not=\emptyset$, 
the set $K$
has size at most $|Y|-1\leq n-1$, 
and if $x_j\in Y$ then for all 
$\ell\in \omega-K$
we have that $t$ is $U_\ell$-absorbing in variable $x_j$.

It follows that if $x_j\in Y$ then
$t$ is $(\bigcap_{j=i}^{\infty} U_j)$-absorbing
in variable $x_j$ for all but finitely many $i$, and hence that
$t$ is $\bigl(\bigcup_{i=0}^{\infty} (\bigcap_{j=i}^{\infty} U_j)\bigr)$-absorbing
in variable $x_j$. This proves that every term operation $t$
of $\m a$ 
has a variable $x_j$ in which $t$ is $U$-absorbing.
By Lemma~\ref{lm-basics}~(5),
this is equivalent to the statement that
all term operations of $\m A$ are compatible with 
all symmetric crosses with base $U$.
\end{proof}

In Theorem~\ref{main1}~(2) we characterized 
the idempotent varieties that fail to have cube terms.
Using Theorem~\ref{main2} we can strengthen this characterization
as follows.

\begin{thm}\label{main3}
Let $\mathcal V$ be an idempotent variety, and let
$\m f=\m f_{\mathcal V}(x,y)$ be the $\mathcal V$-free algebra
over the free generating set $\{x,y\}$.
The following conditions are equivalent.
\begin{enumerate}
\item[{\rm(1)}]
$\mathcal V$ has no cube term.
\item[{\rm(2)}]
$\m f$ has a nonempty proper subuniverse $U$ such that 
$(U,F)$ is a cube term blocker for $\m f$. 
(That is, $U$ is a base for compatible symmetric $d$-ary crosses
of $\m f$ for all $d$.)
\end{enumerate}
\end{thm}

\begin{proof}
Combine Theorem~\ref{main2} for $\m a=\m f$ with Theorem~\ref{main1}.
\end{proof}

As we mentioned in the introduction, cube term blockers were intorduced
in \cite{markovic-maroti-mckenzie} to prove that 
a finite idempotent algebra $\m a$ fails to have a $d$-cube
term for any $d$
if and only if $\m a$ has a cube term blocker
$(U,B)$,
or equivalently,
some subalgebra $\m b$ of $\m a$ has 
compatible symmetric crosses with base $U$
of every arity.

Next we show how to derive this theorem
of \cite{markovic-maroti-mckenzie} 
from the results of our paper.

\begin{lm}\label{hspfin}
In any given signature, the class of idempotent 
algebras with no cube term blockers is closed under the formation
of homomorphic images, subalgebras and finite products.
\end{lm}

\begin{proof}
We prove that the nonexistence of cube term blockers
is preserved under homomorphic images, subalgebras, and
finite products by arguing the contrapositive.
  
If $\varphi\colon \m a\to \m c$ is a surjective homomorphism and
$(U,B)$ is a cube term blocker of $\m c$, then it is easy to see that
$(\varphi^{-1}(U), \varphi^{-1}(B))$ is a cube term
blocker of $\m a$.
This can be done by checking that the inverse image, under $\varphi$,
of each compatible symmetric cross $\cross(U,\dots,U)$
of $\m b$ is a compatible relation of the subalgebra
$\varphi^{-1}(\m B)$ of $\m a$, and is equal to 
the symmetric cross 
$\cross\bigl(\varphi^{-1}(U),\dots,\varphi^{-1}(U)\bigr)$.

If $\m c\leq \m a$ and $(U,B)$ is a cube term
blocker of $\m c$, then it is also a cube term
blocker of $\m a$.

Before turning to products, first note that if $(U,B)$ is a cube
term blocker of $\m a$, and $V$ is a subuniverse of $\m a$,
then $(U\cap V, B\cap V)$ is a cube term blocker for $\m a$
unless $U\cap V = \emptyset$ or $U\cap V=B\cap V$.

Now, to prove the statement for products,
we assume that $(U,B)$ is a cube term blocker of $\m a\times \m c$,
and explain how to find a cube term blocker for either
$\m a$ or $\m c$.
Choose $(a,c)\in B-U$, and let $\pi_{\m a}, \pi_{\m c}$ be
the coordinate projections of $\m a\times \m c$. 
If $\pi_{\m a}(U)\neq \pi_{\m a}(B)$,
then $(\pi_{\m a}(U),\pi_{\m a}(B))$ is a cube blocker of $\m a$, and
we are done. Otherwise $\pi_{\m a}(U)=\pi_{\m a}(B)$,
hence $a\in \pi_{\m a}(B)=\pi_{\m a}(U)$, implying the existence
of an element $(a,d)\in U$.
Now we let $V=\{a\}\times C$ and apply the observation of the
previous paragraph: Since $\emptyset\neq U\cap V\neq B\cap V$,
the pair $(U\cap V, B\cap V)$ is a cube term blocker
for $\m a\times \m c$. 
We further have that
\[
\emptyset\neq\pi_{\m c}(U\cap V)\neq\pi_{\m c}(B\cap V),
\]
since $d\in \pi_{\m c}(U\cap V)$ and 
$c\in \pi_{\m c}(B\cap V)-\pi_{\m c}(U\cap V)$.
Hence 
$(\pi_{\m c}(U\cap C),\pi_{\m c}(B\cap V))$ is a cube term
blocker for $\m c$. 
\end{proof}

\begin{cor}\label{finite-case}
If\/ $\m a$ is a finite idempotent algebra, then the following
are equivalent:
\begin{enumerate}
\item[(1)] $\m a$  has no cube term.
\item[(2)]
${\mathcal V}(\m a)$  has no cube term.
\item[(3)]
$\m f_{{\mathcal V}(\m a)}(x,y)$ has a cube term blocker.
\item[(4)]
$\m a$ has a cube term blocker.
\end{enumerate}
\end{cor}

\begin{proof}
$(1)\Rightarrow(2)$ is clear, because  
a cube term for ${\mathcal V}(\m a)$ would be a cube term for $\m a$.
The implication $(2)\Rightarrow(3)$ 
follows from Theorem~\ref{main3}.
For $(3)\Rightarrow(4)$ we prove the contrapositive.
If $\m a$ has no cube term blocker, then
by Lemma~\ref{hspfin}
and by the fact that $\m f_{{\mathcal V}(\m a)}(x,y)$
lies in $\Hom\!\Sub\!\Prod_{\textrm{fin}}(\m a)$ when $\m a$ is finite,
we get that 
$\m f_{{\mathcal V}(\m a)}(x,y)$ has no cube term blocker.
Finally, $(4)\Rightarrow(1)$ follows from 
Corollary~\ref{comp_crosses_of_cube_terms}. 
\end{proof}

\section{The influence of finite signature}\label{example}

By a \emph{signature} we mean a pair
$\tau=({\mathcal O},\textbf{\sf arity})$ where 
${\mathcal O}$ is a 
set of operation symbols and $\textbf{\sf arity}\colon 
{\mathcal O}\to \omega$
is a function assigning arity. We consider 
only idempotent varieties,
so we may and do consider only signatures
where $\textbf{\sf arity}(f)\geq 2$
for all $f\in {\mathcal O}$. We will call such signatures 
``suitable for idempotent varieties''.

In this section we consider
only finite signatures, which are those where $|{\mathcal O}|<\omega$.
For such a signature $\tau$ we define 
\[
\lvert\tau\rvert = \max_{f\in {\mathcal O}}\bigl(\textbf{\sf arity}(f)\bigr)
\quad\text{and}\quad 
\lVert\tau\rVert = 1+\sum_{f\in {\mathcal O}} \bigl(\textbf{\sf arity}(f)-1\bigr).
\]
It is easy to see that 
$\lvert\tau\rvert = \lVert\tau\rVert$ if there
is only one operation in the signature, 
while $\lvert\tau\rvert < \lVert\tau\rVert$ if
there is more than one.

Now let us consider an idempotent variety
${\mathcal V}$ of finite signature $\tau$,
and let $\m f=\m f_{\mathcal V}(x,y)$ be the $\mathcal V$-free algebra
generated by $\{x,y\}$.
The main results of this section are the following:
\begin{itemize}
\item
If $\m f$ has a compatible cross of arity 
$\lVert\tau\rVert$ or more, then it has compatible crosses 
of all arities. Equivalently, if $\mathcal V$
has a cube term, then it has a
$\lVert \tau\rVert$-cube term (Theorem~\ref{main4}).
\item
For any suitable finite signature $\tau$
there exists an example for $\mathcal V$ where $\m f$ has a compatible
cross of arity 
$\lVert\tau\rVert-1$, but no compatible
crosses of higher arity.
Equivalently, for any suitable finite signature $\tau$
there exists an example for ${\mathcal V}$
such that ${\mathcal V}$ has a $\lVert\tau\rVert$-cube term,
but has no $d$-cube term for $d<\lVert\tau\rVert$
(Example~\ref{exmp1}). 
\end{itemize}
For symmetric crosses the corresponding statements are
as follows:
\begin{itemize}
\item
If $\m f$ has a compatible symmetric cross of arity 
$\lvert\tau\rvert$ or more, then it has compatible 
symmetric crosses 
of all arities (Corollary~\ref{cor-main4}).
\item
For any suitable finite signature
there exists an example where $\m f$ has a compatible
symmetric cross of arity 
$\lvert\tau\rvert-1$, but no compatible
symmetric crosses of higher arity
(Example~\ref{exmp2}).
\end{itemize}

By adding operations to a signature 
one can make $\lVert\tau\rVert$ large while
$\lvert\tau\rvert$ remains small.
Thus one can create varieties with cube terms
where the least dimension of a cube term is much greater
than the arities of the symmetric crosses
of $\m f$.
These results show that we can't use only
symmetric crosses 
to characterize the existence or nonexistence
of cube terms of a fixed dimension.

\begin{thm}\label{main4}
Let $\mathcal V$ be an idempotent variety of finite signature $\tau$.
If\/ $\mathcal V$ has no $\lVert\tau\rVert$-cube term, 
then it has no cube term at all.

In particular, if the signature of $\mathcal V$ consists of a single
binary operation symbol, then either
$\mathcal V$ has a Maltsev term or it has no cube term at all.
\end{thm}

\begin{proof}
  In the second statement, our assumption on the signature $\tau$ forces
  that $\rVert\tau\lVert=2$. Hence the claim is an easy consequence of the
  first statement of the theorem and the fact that a variety has a Maltsev term
  if and only if it has a $2$-cube term.

  To prove the first statement, assume that ${\mathcal V}$
  is an idempotent variety of finite signature $\tau$, and let
  $\m f=\m f_{\mathcal V}(x,y)$.
  We know from Theorem~\ref{main1}~(1) that
  for every $d\ge1$,
  \[
  \text{${\mathcal V}$ has no $d$-cube term}
  \quad\Leftrightarrow\quad
  \text{$\m f$ has a compatible $d$-ary cross.}
  \]
We also know from Theorem~\ref{main3}, from
the definition of a cube term blocker, and from
Corollary~\ref{comp_crosses_of_cube_terms} that
  \[
  \begin{matrix}
    \text{${\mathcal V}$ has no cube term}
    & \stackrel{\rm Thm~\ref{main3}}{\Leftrightarrow}
    & \begin{matrix}
        \text{$\m f$ has a cube term blocker $(U,F)$ where}\\
        \text{$U$ is a nonempty proper subuniverse of $\m f$}
      \end{matrix}\\  
         {\scriptstyle\rm Cor~\ref{comp_crosses_of_cube_terms}}\Uparrow
              \phantom{{\scriptstyle\rm Cor~\ref{comp_crosses_of_cube_terms}}}
         &
         & \phantom{{\scriptstyle\rm{def}}}\Updownarrow{\scriptstyle\rm{def}} \\
    \begin{matrix}
      \text{$\m f$ has compatible crosses}\\
      \text{of all arities}
    \end{matrix}
    & \Leftarrow
    &  \begin{matrix}
         \text{$U$ is the base for compatible symmetric}\\
         \text{crosses of $\m f$ of all arities;}
       \end{matrix}  
 \end{matrix}   
 \]
hence all four conditions displayed here are equivalent.
Therefore it suffices to prove the first statement of Theorem~\ref{main4}
  in the following equivalent formulation.

  \begin{clm}\label{clm-main4}
    Let ${\mathcal V}$ be an idempotent variety
    of finite signature $\tau$,
    and let $\m f=\m f_{\mathcal V}(x,y)$ be the $\mathcal V$-free algebra
    over the free generating set $\{x,y\}$.
    If\/ $\m f$ has a compatible cross of arity $\ge\lVert\tau\rVert$, then
$\m f$ has a nonempty proper subuniverse $U$ such that 
$(U,F)$ is a cube term blocker for $\m f$. 
(That is, $U$ is a base for compatible symmetric crosses
of $\m f$ of all arities.)
  \end{clm}

  \begin{proof}[Proof of Claim~\ref{clm-main4}]
    Assume that
    $\m f$ has a compatible $d$-ary cross $\cross(U_0,\dots,U_{d-1})$
    where $d\ge\lVert\tau\rVert$.
  Let $f_0,\dots,f_{k-1}$ be the operation symbols of $\tau$, and let
  $\textbf{\sf arity}(f_i)=n_i$ ($i<k$). So,
  $d\ge\lVert\tau\rVert=1+\sum_{i=0}^{k-1}(n_i-1)$.
  Applying Lemma~\ref{lm-basics}~(4) to the basic operations
  $f_0,\dots,f_{k-1}$ of $\m f$ we see that for each $i<k$
  there exists a subset $K_i$ of $\ddd:=\mbox{$\{0,\dots,d-1\}$}$ such that
  $|K_i|\le n_i-1$ and $f_i$ has a $U_j$-absorbing variable
  for every $j\in\ddd-K_i$.
  Since $|\ddd-\bigcup_{i=0}^{k-1}K_i|\ge d-\sum_{i=0}^{k-1}(n_i-1)\ge1$,
  we obtain that
  there exists at least one $j\in\ddd$ such that every $f_i$ ($i<k$)
  has a $U_j$-absorbing variable.
  It follows from Lemma~\ref{lm-basics}~(5) that every $f_i$ ($i<k$) is
  compatible with the symmetric crosses $\cross(U_j,\dots,U_j)$ of all
  arities. Hence, the symmetric crosses $\cross(U_j,\dots,U_j)$ of all arities
  are compatible relations of $\m f$.
Equivalently, $(U_j,F)$ is a cube term blocker for $\m f$.
  \renewcommand{\qedsymbol}{$\diamond$} 
  \end{proof}
This finishes the proof of Theorem~\ref{main4}.
\end{proof}

The analog of Claim~\ref{clm-main4} for symmetric crosses can be proved
similarly.

\begin{cor}\label{cor-main4}
  Let $\mathcal V$ be an idempotent variety of finite signature $\tau$, and let
  $\m f=\m f_{\mathcal V}(x,y)$ be the ${\mathcal V}$-free algebra with free
  generators $x,y$.
  If\/ $\m f$ has a compatible symmetric cross $\cross(U,\dots,U)$ of arity
  $\ge\lvert\tau\rvert$, then
  $(U,F)$ is a cube term blocker for $\m f$.
  (That is, $U$ is a base for compatible symmetric crosses
  of $\m f$ of all arities.)
\end{cor}

\begin{proof}
  Assume that $\m f$ has a compatible symmetric $d$-ary cross
  $\cross(U,\dots,U)$ where $d\ge\lvert\tau\rvert$. As before,
  let $f_0,\dots,f_{k-1}$ be the operation symbols of $\tau$, and let
  $\textbf{\sf arity}(f_i)=n_i$ ($i<k$). So,
  $d\ge n_i$ for all $i$. It follows from
  Lemma~\ref{lm-basics}~(4) that every basic operation 
  $f_0,\dots,f_{k-1}$ of $\m f$ has a $U$-absorbing variable.
  Hence, by Lemma~\ref{lm-basics}~(5), $f_0,\dots,f_{k-1}$ are
  compatible with the symmetric crosses $\cross(U,\dots,U)$ of all
  arities. Thus, the symmetric crosses $\cross(U,\dots,U)$ of all arities
  are compatible relations of $\m f$,
or equivalently, $(U,F)$ is a cube term blocker for $\m f$.
\end{proof}  

\begin{exmp}\label{exmp1}
Our goal in this example is to show that the
bound in Theorem~\ref{main4} is sharp. That theorem shows
that if an idempotent variety of finite signature $\tau$
has a $d$-cube term for some $d$, then it has 
one for $d=\lVert\tau\rVert$. 
In this example we construct,
for any suitable finite signature $\tau$,
an idempotent variety that has a $\lVert\tau\rVert$-cube term,
but no $d$-cube term for $d< \lVert\tau\rVert$. 

If one revisits the definition of ``$d$-cube term'',
one sees that the concept of a $1$-cube term is degenerate:
the only varieties with $1$-cube terms are varieties of $1$-element
algebras, and for these varieties any term without nullary
symbols is a $1$-cube term. As noted earlier, a variety
has a $2$-cube term if and only if it has a Maltsev term.
Thus the simplest nondegenerate example to be exhibited
is that of a nontrivial variety with a Maltsev term
in a signature $\tau$ satisfying $\lVert\tau\rVert=2$.
If $\tau$ is suitable for idempotent varieties, then
$\lVert\tau\rVert=2$ implies exactly that
$\tau$ is a signature with one operation, which is binary.
For this signature, take as an example the variety
generated by $\lb \mathbb Z_3; f(x,y)\rangle$ where
$f(x,y)=2x+2y$. This variety is nontrivial, has
$\lVert\tau\rVert=2$, and has 
a Maltsev term $m(x,y,z):=f(f(x,z),y)=x+2y+z$.

The cases where $\lVert\tau\rVert>2$ will be handled
by a uniform construction.
Suppose that $\tau$ has $m$ operation
symbols. Set $n_i=\textbf{\sf arity}(f_i)$, $1\leq i\leq m$, and
set $n:=\lVert\tau\rVert-1=\sum_{i=1}^m (n_i-1)$.
If 
\begin{align*}
C_1&{}=\{1,2,\cdots,(n_1-1)\},\\
C_2&{}=\{(n_1-1)+1, (n_1-1)+2,\cdots,(n_1-1)+(n_2-1)\},\\
 & {}\ \,\vdots\\
C_j&{}=\Bigl\{\Bigl(\sum_{i=1}^{j-1} (n_i-1)\Bigr)+1,\cdots,
                          \Bigl(\sum_{i=1}^{j} (n_i-1)\Bigr)\Bigr\},\\
 & {}\ \,\vdots\\
C_m&{}=\Bigl\{\Bigl(\sum_{i=1}^{m-1} (n_i-1)\Bigr)+1,\cdots,n\Bigr\},
\end{align*}
then 
$\{C_1,\ldots,C_m\}$
is a partition of $[n]:=\{1,\ldots,n\}$ 
whose cells that are in 1--1 correspondence
with the operation symbols: $C_i\leftrightarrow f_i$.
Moreover, $|C_i|=\textbf{\sf arity}(f_i)-1$.
The elements of $[n]$
are going to be coordinates
in a product algebra. To describe the construction
of the algebra we will use the terminology that an element
$j\in [n]$ \emph{belongs to $f_i$},
(or $f_i$ belongs to $j$) if $j\in C_i$.

The universe of the product algebra 
will be $A=\{0,1\}^n$. We will explain how to interpret
each symbol $f_i$ on this set, by describing its
behavior in each coordinate. We need some terminology to do this.
Let $\vee$ (join) and $\wedge$ (meet) be the lattice operations
on $\{0,1\}$ for the order $0<1$. In any given coordinate
we will interpret $f_i$ as either:
\begin{enumerate}
\item[(1)] the \emph{$n_i$-ary join} on $\{0,1\}$: 
\[f_i(x_1,\ldots,x_{n_i}) = \bigvee_{1\leq j\leq n_i} x_j,\]
or
\item[(2)] the \emph{``canonical'' $n_i$-ary
near-unanimity operation}
on $\{0,1\}$,
namely
\[f_i(x_1,\ldots,x_{n_i}) = \bigvee_{1\leq j<k\leq n_i} (x_j\wedge x_k).\]
(This operation is a near-unanimity operation only
when $n_i>2$, but we shall use the terminology even when 
$n_i=2$. In this situation $f_i(x_1,x_2)=x_1\wedge x_2$.)
\end{enumerate}

We interpret $f_i$ on $A = \{0,1\}^n$ by stipulating that
it acts coordinatewise, and that it acts like
the canonical $n_i$-ary
near-unanimity operation
in the coordinates that belong to $f_i$ and 
like the $n_i$-ary join operation
in the coordinates that do not belong to $f_i$.

The set $A$ equipped with the operations 
$f_1,\ldots,f_m$ just defined is the algebra we
call $\m a$.
Each $f_i$ is idempotent on $\m a$, since join, meet
and near-unanimity are idempotent. Therefore, the set
$U_j\subseteq A=\{0,1\}^n$ consisting of all $n$-tuples
with $1$ in the $j$-th coordinate
is a nonempty proper subuniverse of $\m a$
for each $j$ between $1$ and $n$.

\begin{clm}\label{cross_bound}
$\cross(U_1,\ldots,U_n)$ is a compatible $(\lVert\tau\rVert-1)$-ary
cross of $\m a$.
\end{clm}

\begin{proof}[Proof of Claim~\ref{cross_bound}.]
It is only the $n>1$ case of the claim that is interesting,
and we are in that case 
since $\lVert\tau\rVert>2$ and $n=\lVert\tau\rVert-1$.

Elements of $A=\{0,1\}^n$ will be represented
by rows of length $n$ consisting of $0$'s and $1$'s.
The elements of $\cross(U_1,\ldots,U_n)$ are $n$-tuples
of elements of $A$, so could be represented by columns
of length $n$, where each entry in the column is a 
row of length $n$. But instead of doing this, we drop 
parentheses and consider elements of 
$\cross(U_1,\ldots,U_n)$ to be $n\times n$ matrices
of $0$'s and $1$'s. For such a matrix to belong
to this relation we must have the first row
in $U_1$ or the second row in $U_2$, etc.
Since a row of length $n$ belongs to $U_i$
if and only if it has a $1$ in the $i$-th place,
it follows that an $n\times n$ matrix belongs to 
$\cross(U_1,\ldots,U_n)$ exactly if it has a $1$
somewhere on the diagonal.

The operations of $\m a$ act coordinatewise on the columns
in a relation,
and, in a given coordinate, act coordinatewise on rows.
Thus, the operations of $\m a$ act coordinatewise on matrices.
We must show that if $f_i$ is one of the operations
of $\m a$ and $M_1,\ldots, M_{n_i}\in \cross(U_1,\ldots,U_n)$,
then $f_i(M_1,\ldots,M_{n_i})\in \cross(U_1,\ldots,U_n)$.

Suppose that some $M_k$ has a $1$ on its diagonal
in the $j,j$-th entry, where $j$ does not belong
to $f_i$. Then, as $f_i$ acts as $n_i$-ary
join in the $j,j$-position, it follows that 
$f_i(M_1,\ldots,M_{n_i})$ has a $1$ in its $j,j$-th entry, so
$f_i(M_1,\ldots,M_{n_i})\in \cross(U_1,\ldots,U_n)$.
Now suppose that each $M_k$ only has $1$'s in entries
$j,j$ where $j$ does belong to $f_i$. 
Since there are $n_i$ arguments of $f_i$, and only
$n_i-1$ distinct $j$'s that belong to $f_i$, it must be
that there are two matrices $M_k$ and $M_{\ell}$, $1\leq k<\ell\leq n_i$
which both have $1$ in the $j,j$-position for some $j$ belonging
to $f_i$. Since $f_i$ acts like the canonical $n_i$-ary
near unanimity operation in position $j,j$, it follows that
$f_i(M_1,\ldots,M_{n_i})$ has a $1$ in its $j,j$-th entry, so
$f_i(M_1,\ldots,M_{n_i})\in \cross(U_1,\ldots,U_n)$.
These arguments establish the claim.
  \renewcommand{\qedsymbol}{$\diamond$} 
  \end{proof}

\begin{clm}\label{cube_exists}
$\m a$ has a 
cube term.
\end{clm}

\begin{proof}[Proof of Claim~\ref{cube_exists}.]
The fact that the operations of $\m a$
are defined coordinatewise on $\{0,1\}^n$
implies that $\m a$ is a product
of its $2$-element coordinate factor algebras. Thus,
to prove this claim, it is enough to show that each coordinate
factor of $\m a$ has a cube term. 
That this is enough follows from our Corollary~\ref{finite-case}, combined with
Lemma~\ref{hspfin}, 
or from Corollary~2.5 of 
\cite{markovic-maroti-mckenzie}. Namely, each result implies
that if each algebra in a finite
family has a cube term, then the product also has a cube term.

But it is easy to see that each 
coordinate factor of $\m a$ has a cube term
(in fact, a near-unanimity term).
To see this, consider the $j$-th 
coordinate factor algebra 
and suppose that $f_i$ belongs to $j$.
If $\textbf{\sf arity}(f_i)>2$, then $f_i$ interprets
as the canonical $n_i$-ary 
near-unanimity operation in the $j$-th coordinate
and we are done.
If $\textbf{\sf arity}(f_i)=2$, then $f_i$ interprets
as binary meet in the $j$-th coordinate and $f_i$ belongs to
no coordinate other than $j$. Since $n=\lVert\tau\rVert-1>1$,
there exists some coordinate different from $j$.
Hence there must exist some 
$f_k\neq f_i$ belonging to a coordinate
other than $j$. In this case, $f_k$ will interpret
as $n_k$-ary join in the $j$-th coordinate.
With join coming from $f_k$ and meet coming from
$f_i$ one can construct a ternary
near-unanimity operation in coordinate $j$.
  \renewcommand{\qedsymbol}{$\diamond$} 
  \end{proof}

Claim~\ref{cube_exists} shows that $\m a$ has a 
$d$-cube term for some $d$.
Hence, we can use Theorem~\ref{main4} to conclude that
$\m a$ has a $\lVert\tau\rVert$-cube term.
On the other hand,  
by Corollary~\ref{comp_crosses_of_cube_terms},
Claim~\ref{cross_bound} prevents $\m a$ from having
a $d$-cube term for any $d<\lVert\tau\rVert$.
This proves all required properties of $\m a$.
\hfill\qed
\end{exmp}

\begin{remark}
  Example~\ref{exmp1} was discovered
  with the help of UACalc, a universal algebra calculator.
  After including it here we learned that the preprint
  \cite{ca-co-valeriote} by
  Campenella, Conley and Valeriote
  contains essentially the same example. We are informed
  that they 
  also discovered the example with the help of UACalc. The purpose
  of the example in their paper is roughly the same as ours
  (i.e., lower bounds for dimension estimates),
  except our application is to cube terms and their
  application is to near unanimity terms.
  \end{remark}

\begin{exmp}\label{exmp2}
  In this example our goal is to show that the bound in
  Corollary~\ref{cor-main4} is sharp. Accordingly, we want to construct, 
  for any suitable finite signature $\tau$, an
  idempotent variety $\mathcal{V}$ such that the free algebra
  $\m f=\m f_{\mathcal{V}}(x,y)$ has a compatible symmetric cross of arity
  $|\tau|-1$, but no compatible symmetric crosses of higher arity.

  If $\tau$ is suitable for idempotent varieties and $|\tau|=2$, then
  $\tau$ is a signature with binary operation symbols only,
  say $f_1,\dots,f_m$. In this case
  we can choose $\mathcal{V}$ to be the variety generated by an algebra
  $\langle\mathbb{Z}_3;f_1,\dots,f_m\rangle$ where
  $f_1(x,y)=2x+2y$ and for $2\le i\le m$, $f_i$ is interpreted as a
  projection.
  Since $\mathcal{V}$ has a Maltsev term,
  $\m f$ has no compatible symmetric crosses of arity $\ge 2$. However,
  $\cross(\{x\})$ is a compatible symmetric cross of $\m f$ of arity $1$.

  Let us assume from now on that $|\tau|\ge3$, and let $f_1,\dots,f_m$
  denote the operation symbols in $\tau$ where $f_i$ is $n_i$-ary
  ($1\le i\le m$). We may assume without loss of generality that
  $n_1\ge \dots\ge n_m$, so $|\tau|=n_1\,(\ge3)$.
  Now consider an algebra $\m b=\langle\{0,1\};f_1,\dots,f_m\rangle$
  where $f_1$ is interpreted on $\{0,1\}$ as the
  ``canonical'' $n_1$-ary near-unanimity operation (see the definition
  in Example~\ref{exmp1}), and for $2\le i\le m$,
  $f_i$ is interpreted as a projection.
  Let $\mathcal{V}$ be the variety generated by $\m b$, and let
  $\m f=\m f_{\mathcal{V}}(x,y)$. Since $\mathcal{V}$ has an $n_1$-ary
  near-unanimity operation, we know from
  Corollary~\ref{comp_crosses_of_cube_terms} that
  $\m f$ has no compatible crosses of arity $\ge n_1=|\tau|$.
  On the other hand, since the $(n_1-1)$-ary cross
  $\cross(\{1\},\dots,\{1\})=\{0,1\}^{n_1-1}\setminus\{(0,\dots,0)\}$
  is a compatible relation of $\m b$, its inverse image
  under the homomophism $\m f\to\m b$ sending $x$ to $1$ and $y$ to $0$ 
  is a compatible symmetric cross of $\m f$ of arity $n_1-1=|\tau|-1$.
\end{exmp}

\section{A fact about cyclic term varieties}\label{bergman-section}

This note emerged in response to a question
we learned from Cliff Bergman:
Suppose that ${\mathcal C}_2$ is the variety 
defined with one binary operation and axiomatized by
the identities
\begin{enumerate}
\item[(1)] $w(x,x)=x$, and
\item[(2)] $w(x,y)=w(y,x)$.
\end{enumerate}
Is it true that every subvariety of ${\mathcal C}_2$ either contains the 
$2$-element semilattice or is congruence permutable?

Bergman's question arose out of a certain line of 
investigation into constraint satisfaction problems.
Namely, it is of interest to understand whether
the algebras in the \emph{pure cyclic term} varieties
have tractable CSP's. The $d$-ary pure cyclic term 
variety ${\mathcal C}_d$ is defined with one $d$-ary operation
satisfying
\begin{enumerate}
\item[(1)] $w(x,x,\ldots,x)=x$, and
\item[(2)] $w(x_1,x_2,\ldots,x_d)=w(x_2,x_3,\ldots,x_1)$.
\end{enumerate}
If one could show that each finite algebra in each 
variety ${\mathcal C}_d$ has tractable associated CSP's,
then one would have solved the Feder--Vardi Conjecture.

Bergman and David Failing showed in \cite{bergman-failing}
that if $\mathcal V$ is a subvariety of 
${\mathcal C}_2$
that is the join of a congruence permutable variety
and the variety of semilattices, 
then the finite algebras in $\mathcal V$ have tractable 
associated CSP's.
So, Bergman was really asking whether this theorem
applied to every subvariety of ${\mathcal C}_2$
that is a join of the variety of semilattices
and a disjoint subvariety. When Bergman asked the question,
he mentioned that an affirmative answer was supported
by extensive computer computations performed by 
Bergman, William DeMeo, and Jiali Li.

Here we explain why the answer to Bergman's question is affirmative,
even with $2$ replaced by $d$. That is, we explain why a subvariety
of the $d$-ary pure cyclic term variety either contains
a 2-element semilattice or else has a $d$-cube term.
For this, define 
a \emph{$2$-element semilattice in ${\mathcal C}_d$}
to be an algebra with 2-element universe, say $\{0,1\}$
and operation defined by 
\begin{eqnarray}\label{2semi}
w(x_1,\ldots,x_d) = 
\begin{cases}
1& \mbox{\rm if}\; x_1=x_2=\cdots=x_d=1;\\
0& \mbox{\rm else.}
\end{cases}
\end{eqnarray}
More generally, a $2$-element semilattice for a
variety $\mathcal V$ is one in which, for every $d$,
each $d$-ary fundamental
operation satisfies (\ref{2semi}).

\begin{thm}\label{pure_cyclic}
A subvariety of the pure $d$-ary cyclic term variety
either has a $d$-cube term or contains a
$2$-element semilattice.
A finite algebra in the pure $d$-ary cyclic term variety
either has a $d$-cube term or has a
$2$-element semilattice
section.
\end{thm}

One should note that a $2$-element semilattice  has no
$d$-cube term for any $d$, so the
two cases described in the theorem
are complementary.

\begin{proof}
  Let $\mathcal V$ be a subvariety of the pure $d$-ary cyclic term variety.
  If $\mathcal V$ does not have a $d$-cube
  term, then by Theorem~\ref{main1}~(1) the free algebra
  $\m f = \m f_{\mathcal V}(x,y)$ has a compatible
  $d$-ary cross. But the signature $\tau$ of $\mathcal V$ satisfies 
  $d = \|\tau\|$, so Claim~\ref{clm-main4} shows that
  $\m f$ has a cube term blocker of the form $(U,F)$.
  It follows from Lemma~\ref{lm-basics}~(5) that
  the cyclic term of the variety must be $U$-absorbing
  in at least one of its variables, so by cyclicity
  this term is $U$-absorbing in all of its variables.
  This implies that
    (i)~the congruence $\theta$ on $\m f$ generated by $U\times U$
    is the union of $U\times U$ and the equality relation,
    and (ii)~if 
  $t\in F\setminus U$ is chosen arbitrarily, then
  $S = U\cup \{t\}$ is a subuniverse of $\m f$.
  The algebra $\m s/\theta|_S$ must then be a $2$-element
  semilattice in $\mathcal V$.

  For the second statement,
  assume that $\m a$ is a finite algebra in the pure $d$-ary
  cyclic term variety and that $\m a$ does not have a $d$-cube term.
  Then ${\mathcal V}(\m a)$ cannot have a
$d$-cube term, so by Theorem~\ref{main4}, 
${\mathcal V}(\m a)$ cannot have a cube term at all.
Therefore Corollary~\ref{finite-case} guarantees that $\m a$ 
has a cube term blocker, say $(U,B)$. Now construct a $2$-element
  semilattice from this blocker in the same manner
  one was constructed from the blocker $(U,F)$ in the preceding
  paragraph. It will be a section of $\m a$.
\end{proof}

What matters to us in Theorem~\ref{pure_cyclic}
is that the $d$ in ``$d$-ary cyclic term''
agrees with the $d$ in 
``$d$-cube term''; that is, the theorem
establishes the existence of a cube
term under some condition, \emph{and bounds its index}.
If one is not concerned with such a bound, then
one can establish a result about varieties with many
cyclic fundamental operations, namely:

\begin{thm}
  Let $\mathcal V$ be an
  idempotent variety whose fundamental operations
  each satisfy cyclic identities. That is, for each
  fundamental operation $w(x_1,\ldots,x_n)$ it is the case that
${\mathcal V}\models w(x_1,x_2,\ldots,x_n)=w(x_2,x_3,\ldots,x_1).$
  Then $\mathcal V$ either has a
  cube term or contains a 2-element semilattice.
\end{thm}

\begin{proof}
  Let $\m f = \m f_{\mathcal V}(x,y)$.
  If $\mathcal V$ has no cube term, then by
  Theorem~\ref{main3} there is a cube term
  blocker $(U,F)$ for $\m f$. By repeating the
  argument in the first paragraph of the proof
  of Theorem~\ref{pure_cyclic} we find
  that $\mathcal V$ contains a $2$-element semilattice.
    \end{proof}

\section{Generic crosses}\label{generic-section}

In this section we focus on idempotent varieties of finite type.
We show that if a nontrivial member of such
a variety
has a compatible $d$-ary cross, then some countably infinite
algebra $\m a$ in the variety
has a `generic' compatible $d$-ary cross.
By a `generic cross' we mean a cross
$\cross(U_1,\ldots,U_n)$ where the
sets $U_1,\ldots,U_n$ are as independent as possible.
Specifically, when $\cross(U_1,\ldots,U_n)$
is a cross on a countably infinite set $A$, 
we call $\cross(U_1,\ldots,U_n)$ 
\emph{generic}
if every nonzero Boolean combination of the 
sets $U_1,\ldots,U_n$ is countably infinite.

\begin{thm}\label{prime_thm}
  If $\mathcal X$ is an idempotent variety of finite
  type and some member of $\mathcal X$ has a compatible
  $d$-ary cross, then some countably infinite member
  of $\mathcal X$ has a compatible $d$-ary generic cross.
\end{thm}

\begin{proof}
  If some member
  of $\mathcal X$ has a compatible
  $d$-ary cross, then
  by Corollary~\ref{comp_crosses_of_cube_terms},
  $\mathcal X$ cannot have a $d$-cube term.
  Hence Theorem~\ref{main3} implies that
the
algebra $\m f = \m f_{\mathcal X}(x,y)$ must
have a compatible
$d$-ary cross, say
  \[
  \cross(U_1,\ldots,U_d) = B_1\cup\cdots\cup B_d
  \]
  where $B_i = F\times \cdots\times F\times U_i\times F\times\cdots\times F$
  is full in all coordinates except the $i$th.
  Here $\m f$ need not be
  infinite, and this cross need not be generic, so we modify
  the situation as follows.

  Let $\m a\in\mathcal X$ be a countably infinite algebra.
  Now define
  \begin{align*}
\mathcal F &{}= \m f^d\times \m a,\\
{\mathcal U}_1 &{}= B_1\times A,\\
&\,\,\,\vdots\\
{\mathcal U}_d&{}= B_d\times A.
    \end{align*}
  It is easy to see that
  $\mathcal F$ is countably
  infinite and that each
  ${\mathcal U}_i$ is a
  nonempty
  proper subuniverse of
  $\mathcal F$.

  \begin{clm}\label{clm-generic}
    $\cross({\mathcal U}_1,\ldots,{\mathcal U}_d)$
    is a compatible generic $d$-ary cross
    of $\mathcal F$.
    \end{clm}

  \begin{proof}[Proof of Claim~\ref{clm-generic}]    
  We first argue that $\cross({\mathcal U}_1,\ldots,{\mathcal U}_d)$
    is compatible, i.e.\ a subuniverse of $\mathcal F^d$.
    For this we consider ${\mathcal F} = \m f^d\times \m a$
    to have $d+1$ coordinates, so
    \[
      {\mathcal F}^d = \m f^d\times \m a \times
      \m f^d\times \m a \times\cdots\times \m f^d\times \m a
      \]      
    has $d(d+1)$ coordinates. Notice that all coordinate
    algebras in this direct representation of ${\mathcal F}^d$
    are $\m f$ except those whose coordinates lie in the arithmetical
    progression $d+1, 2(d+1), \cdots, d(d+1)$,
    in which case the coordinate algebras
    are $\m a$.

    There is a projection homomorphism
    $\pi\colon {\mathcal F}^d\to \m f^d$
    which projects onto the coordinates in the arithmetic progression
    $1, (d+1)+2, 2(d+1)+3, \cdots, (d-1)(d+1)+d$, which projects
    onto the first coordinate of the first block of $d+1$
    factors of ${\mathcal F}^d$, the second factor of the second
    block of $d+1$ factors, etc. It is not hard to verify that 
    \[
    \cross({\mathcal U}_1,\ldots,{\mathcal U}_d)
    = \pi^{-1}\bigl(\cross(U_1,\ldots,U_d)\bigr),
    \]
thereby establishing that $\cross({\mathcal U}_1,\ldots,{\mathcal U}_d)$
 is compatible.

  To show that
  $\cross({\mathcal U}_1,\ldots,{\mathcal U}_d)$
  is generic, it suffices to show
  that any intersection
  $
  {\mathcal U}_1^{\varepsilon_1}\cap \cdots \cap {\mathcal U}_d^{\varepsilon_d}
  $
  contains infinitely many elements,
  where $\varepsilon_i=\pm 1$ for each $i$
  and ${\mathcal U}_1^{+1}={\mathcal U}_1$ while
  ${\mathcal U}_1^{-1}={\mathcal F}\setminus {\mathcal U}_1$.
  Observe that a $(d+1)$-tuple $(u_1,\ldots,u_d,a)$
  belongs to the set
  ${\mathcal U}_1^{\varepsilon_1}\cap \cdots \cap {\mathcal U}_d^{\varepsilon_d}$
  exactly when 
  $u_i\in {\mathcal U}_i$ if $\varepsilon_i=+1$, 
  $u_i\in {\mathcal F}\setminus {\mathcal U}_i$ if $\varepsilon_i=-1$,
  and $a\in A$. Such choices are possible
  since ${\mathcal U}_i$ is a
  nonempty proper subuniverse of
$\mathcal F$ for each $i$ and $A$ is nonempty.
  If we fix the $u_i$'s and let the last coordinate $a$
range over the infinite set $A$ we obtain infinitely
many elements in 
  ${\mathcal U}_1^{\varepsilon_1}\cap \cdots \cap {\mathcal U}_d^{\varepsilon_d}$.
\renewcommand{\qedsymbol}{$\diamond$} 
  \end{proof}
  This completes the proof of Theorem~\ref{prime_thm}.
\end{proof}

\begin{cor}\label{prime_cor}
  The class of varieties having a $d$-cube term
  represents a join-prime filter in the lattice
  of idempotent Maltsev conditions.
\end{cor}

\begin{proof}
  If not, then there are idempotent varieties
  $\mathcal X$ and $\mathcal Y$ that do not
  have a $d$-cube term, but their coproduct
  $\mathcal X\sqcup \mathcal Y$
  does have a $d$-cube term. But if
  $\mathcal X\sqcup \mathcal Y$ has a $d$-cube term,
  then so must $\mathcal X'\sqcup \mathcal Y'$
  for some finitely presentable varieties
  $\mathcal X'$ interpretable in $\mathcal X$ and 
  $\mathcal Y'$ interpretable in $\mathcal Y$.
  Replacing $\mathcal X$ and $\mathcal Y$
  by $\mathcal X'$ and $\mathcal Y'$ we may assume that
  $\mathcal X$ and $\mathcal Y$ are finitely presentable,
  in particular of finite type.

  We prove the corollary by arguing
  that if $\mathcal X$ and $\mathcal Y$ have finite
  type and no $d$-cube term, then $\mathcal X\sqcup\mathcal Y$
  has no $d$-cube term.

  By Theorem~\ref{prime_thm} there exist countably infinite algebras
$\m a\in \mathcal X$ and
  $\m b\in \mathcal Y$ which have generic 
  compatible $d$-ary crosses, say $\cross(U_1,\ldots,U_d)$
  and $\cross(V_1,\ldots,V_d)$. By genericity, it is possible to
  find a bijection $\alpha\colon A\to B$ such that
  $\alpha(U_i)=V_i$ for all $i$.
  There is a unique $\mathcal X$-structure $\m b'$ on $B$
  that makes $\alpha\colon \m a\to \m b'$ an isomorphism.
Thus
  $\cross(V_1,\ldots,V_d) = \alpha(\cross(U_1,\ldots,U_d))$
  is a compatible cross of $\m b'$. Since it is also a compatible
  cross of $\m b\in \mathcal Y$, the algebra
  on $B$ obtained by merging $\m b$ and $\m b'$
  is a model of the identities of $\mathcal X\sqcup \mathcal Y$
  which has a compatible $d$-ary cross. This is enough to show that
  $\mathcal X\sqcup\mathcal Y$ has no $d$-cube term.
  \end{proof}

\begin{remark}
  The results in this
  section were discovered during the 2016
  `Algebra and Algorithms' workshop after hearing
  a talk by Matthew Moore on the join-primeness
  among \emph{idempotent linear} Maltsev conditions
    of the condition expressing
the existence of a cube term.
Later, Jakub Opr\v{s}al pointed us to his preprint
\cite{oprsal} where he proves our
Corollary~\ref{prime_cor} (among other things).
Opr\v{s}al told us that he learned of our
characterization of cube terms in terms of crosses
from a talk of Szendrei at the AAA90 conference
in Novi Sad in 2015, and then developed a similar
characterization of his own which allowed
him to prove
Corollary~\ref{prime_cor}. His discovery
of Corollary~\ref{prime_cor} predates ours.
His argument depends on an analogue of Theorem~\ref{prime_thm},
which he proves for varieties in arbitrary languages.
Our proof also works for arbitrary languages,
but we decided not to change ours after learning
about Opr\v{s}al's work.
  \end{remark}

\bibliographystyle{plain}

\begin{thebibliography}{10}

\bibitem{aichinger-mayr-mckenzie}
Aichinger, Erhard; Mayr, Peter; McKenzie, Ralph 
{\it On the number of finite algebraic structures.} 
J.\ Eur.\ Math.\ Soc.\ (JEMS) {\bf 16} (2014), no.~8, 1673--1686. 

\bibitem{barto}
Barto, Libor
{\it  Finitely related algebras in congruence modular varieties have
  few subpowers.}
J.\ Eur.\ Math.\ Soc.\ (JEMS), to appear.

\bibitem{bergman-failing}
Bergman, Clifford; Failing, David
{\it Commutative, idempotent groupoids and the constraint
  satisfaction problem.}
 Algebra Universalis {\bf 73} (2015), no.~3-4, 391--417.

\bibitem{berman-idziak-markovic-mckenzie-valeriote-willard}
Berman, Joel; Idziak, Pawe{\l}; Markovi\'c, Petar; 
McKenzie, Ralph; Valeriote, Matthew; Willard, Ross 
{\it Varieties with few subalgebras of powers.} 
Trans.\ Amer.\ Math.\ Soc.\ {\bf 362} (2010), no.~3, 1445--1473.

\bibitem{bulatov-mayr-szendrei} 
Bulatov, Andrei; Mayr, Peter; Szendrei, \'Agnes
{\it The subpower membership problem for finite algebras with cube terms.}
Preprint.

\bibitem{ca-co-valeriote}
Campanella, Maria; Conley, Sean; Valeriote, Matt
{\it Preserving near unanimity terms under products.}  
Algebra Universalis, to appear.

\bibitem{freese-kiss-valeriote}
Freese, Ralph; Kiss, Emil; Valeriote, Matthew
{\it Universal {A}lgebra {C}alculator},
{Available at: {\verb+www.uacalc.org+}}, {2011}.

\bibitem{idziak-markovic-mckenzie-valeriote-willard}
Idziak, Pawe{\l}; Markovi\'c, Petar; 
McKenzie, Ralph; Valeriote, Matthew; Willard, Ross 
{\it Tractability and learnability arising from algebras with few subpowers.} 
SIAM J.\ Comput.\ {\bf 39} (2010), no. 7, 3023--3037.

\bibitem{kearnes-szendrei2012}
Kearnes, Keith A.; Szendrei, \'Agnes 
{\it Clones of algebras with parallelogram terms.} 
Internat.\ J.\ Algebra Comput. {\bf 22} (2012), no.~1, 1250005, 30 pp. 

\bibitem{kearnes-szendrei}
Kearnes, Keith A.; Szendrei, \'Agnes 
{\it Dualizable algebras with parallelogram terms.}
Algebra Universalis, to appear.
http://arxiv.org/abs/1502.02192

\bibitem{kearnes-tschantz}
Kearnes, Keith A.; Tschantz, Steven T. 
{\it Automorphism groups of squares and of free algebras.} 
Internat.\ J.\ Algebra Comput.\ {\bf 17} (2007), no.~3, 461--505.

\bibitem{markovic-maroti-mckenzie}
Markovi\'c, Petar; Mar\'oti, Mikl\'os; McKenzie, Ralph 
{\it Finitely related clones and 
algebras with cube terms.} Order {\bf 29} (2012), no.~2, 345--359.

\bibitem{moore}
Moore, Matthew
{\it Naturally dualisable algebras omitting types 
  {\bf 1} and {\bf 5} have a cube term.}
Algebra Universalis {\bf 75} (2016), 221--230.

\bibitem{oprsal}
Opr\v{s}al, Jakub
{\it Taylor's modularity conjecture and related problems for
  idempotent varieties.}  \hfill \\
http://arxiv.org/abs/1602.08639v1

\bibitem{szendrei87}
Szendrei, \'Agnes
{\it Idempotent algebras with restrictions on subalgebras.}
Acta Sci.\ Math.\ (Szeged) {\bf 51} (1987), no.~1-2, 251--268. 

\end{thebibliography}

\end{document}